\documentclass[a4paper,10pt]{article}
\usepackage{amsfonts}
\usepackage{amssymb,amsthm,color,tikz}
\usepackage{amsmath}
\usepackage{comment}
\usepackage{dsfont}

\usepackage[normalem]{ulem}
\usepackage{url}

\newcommand{\R}{\mathbb{R}}
\newcommand{\N}{\mathbb{N}}

\newcommand{\Z}{\mathbb{Z}}

\newcommand{\Pa}{\mathbb{P}}

\newtheorem{theorem}{Theorem}
\newtheorem{lemma}[theorem]{Lemma}
\newtheorem{proposition}[theorem]{Proposition}

\newcommand{\Om}{\Omega}
\newcommand{\lb}{\lambda}

\newcommand{\dd}{\mathrm{d}}

\newcommand{\er}{\mathrm{e}}

\newcommand{\bp}{\begin{proof}}
\newcommand{\ep}{\end{proof}}

\begin{document}
\title{On the gradient of the torsion function for elongated cylinders}
\author{{M. van den Berg} \\
School of Mathematics, University of Bristol\\
Fry Building, Woodland Road\\
Bristol BS8 1UG, United Kingdom\\
\texttt{mamvdb@bristol.ac.uk}}
\date{30 June 2026}\maketitle
\vskip 1.5truecm \indent

\begin{abstract}\noindent
Let $(0,L)\times E_a\subset \R^{d+1}$ denote the cylinder of length $L$, and base $E_a\subset \R^{d},$ where $E_a$ is an open ellipsoid with semi-axes $a=(a_1,...,a_d)$. Let $w_{(0,L)\times E_a}$ denote its torsion function. Heat equation tools are used to show that (i) if $E_a$ is has small eccentricity and $L$ is sufficiently large, then the maxima of $|\nabla w_{(0,L)\times E_a}|$ are located at the centres $(0,0)$ and $(L,0)$ respectively, (ii) if $E_a$ has large eccentricity and $L$ is sufficiently large, then the maxima are located on the lateral side of the cylinder.
\end{abstract}
\vskip 1truecm \noindent   {Mathematics Subject
Classification ({2020})}: { 28A75, 35B50, 35K05, 60J65}.

\section{Introduction and main results \label{sec1}}

In this paper we obtain results for the torsion function for non-empty open sets $\Om\subset \R^d$ with boundary $\partial\Om$. Consider the elliptic partial differential equation
\begin{equation}\label{e7}
-\Delta w=1,\,\,w\big|_{\partial\Om}=0.
\end{equation}
Let $\lb(\Om)$ denote the bottom of the spectrum of the Dirichlet Laplacian acting in $L^2(\Om)$. That is
\begin{equation*}
\lb(\Om)= \inf\Big\{\int_\Om |\nabla u|^2: u\in C_c^{\infty},\,\int_\Om u^2=1\Big\}.
\end{equation*}
It is well known that if $\lb(\Om)>0$ then \eqref{e7} has a unique solution, the torsion function, denoted by $w_{\Om}$. The torsion function is non-negative, bounded, and satisfies
\begin{equation*}
1\le \|w_{\Om}\|_{\infty}\lambda(\Om)\le c_d,
\end{equation*}
where $c_d$ depends on $d$ only. See for example \cite[Theorem 1]{vdBC}, and \cite[Theorem 1.5]{HV}.
The Dirichlet boundary condition $w_{\Om}\big|_{\partial\Om}=0$ in \eqref{e7} holds at all regular points of $\partial\Om$ \cite[Section 2]{BR}.

\smallskip

The literature on the $L^{\infty}$ norm of $|\nabla w_{\Om}|$ is extensive (see e.g. \cite{B, P, HS, SL,Stb} and the references therein). Since $|\nabla w_{\Om}|^2$ is subharmonic and $w_{\Om}|_{\partial\Om}=0$,  $|\nabla w_{\Om}|$ does not have a maximum in $\Om$ (see e.g. \cite{GP}, \cite[p.85]{RS}). If $\partial\Om$ is smooth or piecewise smooth then $|\nabla w_{\Om}|$ has a maximum on $\partial \Om$.
The location of the maxima of $|\nabla w_{\Om}|$ has been of great interest in particular for $\Om$ planar and convex \cite{BK,LY,LX}.
The points at which $|\nabla w_{\Om}|$ is maximal are of physical significance: $\nabla w_{\Om}$ contains the non-vanishing stress components of an elastic (infinitely long) cylindrical bar with cross-section $\Om$. These points of maximal stress are also known as fail points  \cite[p.195]{BK}. However, even in this simple setting of $\Om\subset \R^2$ there is no complete picture  of the location of these fail points \cite{LY,LX}.
Before we state our main result, Theorem \ref{the1} below, we recall two classical results, and introduce some further notation.

\smallskip

It was shown in \cite[(6.12)]{RS} that for $\Om\subset\R^d$ open and convex,
\begin{equation}\label{e10}
\|\nabla w_{\Om}\|^2_{\infty}\le 2\|w_{\Om}\|_{\infty}.
\end{equation}
Let $h>0$, and let
\begin{equation*}
\Om_h=\Big\{(x_1,x')\in \R^d: -\frac{h}{2}<x_1<\frac{h}{2},\,x'\in \R^{d-1}\Big\}.
\end{equation*}
The torsion function for $\Om_h$ is given by
\begin{equation*}
w_{\Om_h}(x)=\frac12\Big(\frac{h^2}{4}-x_1^2\Big),\,x\in \Om_h.
\end{equation*}
This gives equality in \eqref{e10} for $\Om=\Om_h.$

\smallskip

It was shown in \cite[(3),(4)]{WH} that if $w_{\Om}$ is the torsion function for an open set $\Om\subset \R^d$, then
\begin{equation}\label{e13}
\Delta\big(|\nabla w_{\Om}|^2+\frac2d w_{\Om}\big)\ge 0.
\end{equation}
This improves the results mentioned above on subharmonicity of $|\nabla w_{\Om}|^2$. Furthermore if $\partial\Om$ is regular and if $w_{\Om}\in L^{\infty}(\Om)$, then \eqref{e13} implies that
\begin{equation}\label{e14}
\|\nabla w_{\Om}\|^2_{\infty}\ge \frac2d\|w_{\Om}\|_{\infty}.
\end{equation}
This complements the upper bound \eqref{e10}.

\smallskip

Let
\begin{equation*}
E_a=\Big\{x\in \R^{d}: \sum_{i=1}^{d} \Big(\frac{x_i}{a_i}\Big)^2<1\Big\}.
\end{equation*}
denote the ellipsoid with semi-axes $a=\{a_1,a_2,...,a_{d}\}$ where  $a_i>0,\,i=1,...,d$.
The torsion function for $E_a$ is given by
\begin{equation*}
w_{E_a}(x)=\frac12\Big(\sum_{i=1}^{d} \Big(\frac{1}{a_i}\Big)^2\Big)^{-1}\Big(1-\sum_{i=1}^{d} \Big(\frac{x_i}{a_i}\Big)^2\Big),\,x\in E_a.
\end{equation*}
We put
\begin{equation*}
a_{-}=\min\{a_1,a_2,...\},\,a_{+}=\max\{a_1,a_2,...\},
\end{equation*}
and denote the eccentricity of $E(a)$ by
\begin{equation*}
\mathfrak{e}(E(a))=\frac{a_{+}}{a_{-}}-1,
\end{equation*}

Let $\tilde{x}'\in\partial E_a$ be such that $|\tilde{x}'|=a_{-}$. Note that
\begin{equation}\label{e18}
\|\nabla w_{E_a}\|_{\infty}=|\nabla w_{E_a}|(\tilde{x}')=|\nabla w_{E_a}|(-\tilde{x}')=\frac{1}{a_{-}}\Big(\sum_{i=1}^{d} \Big(\frac{1}{a_i}\Big)^2\Big)^{-1}<a_{-}.
\end{equation}
Let $B_d=\{y\in\R^d:|y|<1\}=E_{(1,...,1)}$. Then
\begin{equation*}
w_{B_d}(x)=\frac{1}{2d}(1-|x|^2),\,x\in B_d,
\end{equation*}
which implies equality in \eqref{e14} for $\Om=B_d$.

\smallskip

For a non-empty open set $\Om$ we let $v_{\Om}$ be the solution of
\begin{equation}\label{e3}
\begin{cases}
\frac{\partial v}{\partial t}=\Delta v,\qquad &x\in \Om, t>0,\\
\left.v\right|_{\partial\Om}=0,\qquad&t>0,\\
\left.v\right|_{t=0}=1,\qquad&x\in\Om.
\end{cases}
\end{equation}

The following quantity shows up in the statement of Theorem \ref{the1}(i). Let $d\ge 1$, and let
\begin{equation}\label{e21}
I_d=\int_{(0,\infty)}\frac{\dd t}{(\pi t)^{1/2}}v_{B_d}(0;t).
\end{equation}

\smallskip

In Theorem \ref{the1} below we consider the case of a finite cylinder with length $L$ where the base is an ellipsoid. We show that if $d=2$, if $\mathfrak{e}(E(a))$ is sufficiently small, and if the length of the cylinder is sufficiently large, then the maxima of the length of the gradient of the torsion function are at the centre of the base and at the centre of the top of the ellipsoid. If on the other hand $\mathfrak{e}(E(a))$ is large, and $L$ is sufficiently large, then the maxima of the gradient of the torsion function are on the lateral side of the cylinder.
\begin{theorem}\label{the1}
\begin{itemize}
\item[\textup{(i)}] If $L>0,\,d\ge 2$, and if
\begin{equation}\label{e27}
1+\frac{2a_{-}}{L}<I_d\sum_{i=1}^d\frac{a_{-}^2}{a_i^2},
\end{equation}
then
\begin{align}\label{e27a}
\|\nabla w_{(0,L)\times E_a}\|_{\infty}&=|\nabla w_{(0,L)\times E_a}|(0,0)=|\nabla w_{(0,L)\times E_a}|(L,0)\nonumber\\&
>|\nabla w_{(0,L)\times E_a}|(x_1,x'),\,\, \,0<x_1<L,\,x'\in \partial E_a.
\end{align}

\smallskip

\item[\textup{(ii)}] If $L>0,\,d\ge 2$, and if
\begin{equation}\label{e28}
\Big(\frac{2^{7/2}}{\pi}e^{-\pi L/(8a_{-})}+\frac{8G}{\pi^2}\Big)\sum_{i=1}^d\frac{a_{-}^2}{a_i^2}<1,
\end{equation}
then
\begin{align}\label{e28a}
\|\nabla w_{(0,L)\times E_a}\|_{\infty}&\ge |\nabla w_{(0,L)\times E_a}|(L/2,\tilde{x}')=|\nabla w_{(0,L)\times E_a}|(L/2,-\tilde{x}')\nonumber\\&>
|\nabla w_{(0,L)\times E_a}|(0,x')=|\nabla w_{(0,L)\times E_a}|(L,x'),\,\,\,x'\in E_a,
\end{align}
where $G=0.915965...$ is Catalan's constant defined by
\begin{equation}\label{e28b}
G=\sum_{k=0}^\infty\frac{(-1)^k}{(2k+1)^2}.
\end{equation}
\end{itemize}
\end{theorem}

\medskip

In Proposition \ref{prop1} below we show that $I_d$ defined in \eqref{e21} is finite for all $d=1,2,...$, and that $dI_d>1$ if and only if $d\ge 2$. The latter guarantees that, for $d\ge 2$, the collection of ellipsoids in \eqref{e27} is non-empty.

Requirement \eqref{e27} can only hold if $\mathfrak{e}(E(a))$
and $a_{-}/L$ are sufficiently small: that is $L$ is sufficiently large, and $\mathfrak{e}(E_a)$ is small. In this case the maxima of $|\nabla w_{(0,L)\times E_a}|$ are at $(0,0)$ and $(L,0)$.

Requirement \eqref{e28} can only hold if $\sum_{i=1}^d\frac{a^2_{-}}{a_i^2}<\frac{\pi^2}{8G}$ and $a_{-}/L$ is sufficiently small. This collection of ellipsoids is non-empty since $\frac{\pi^2}{8G}>1$. In this case the maxima of $|\nabla w_{(0,L)\times E_a}|$ are on the lateral side of boundary of the cylinder. We note that \eqref{e28} is never satisfied if there are two-semi axes with length $a_{-}$.

If $d=1$ then the ellipsoid is just an interval say $(-a,a)$, and $(0,L)\times (-a,a)$ is a rectangle. Since $a=a_{-}=a_{+}$  \eqref{e27} is not satisfied for $d=1$ since $I_1< 1$. It was shown in \cite[pp.275-277]{TG} that the maxima of $|\nabla w_{(0,L)\times (-a,a)}|$ are at the mid-points of the long sides of the rectangle. If $L=2a$ then maxima are at the four midpoints of the sides of the square.

We do not have such a complete picture for ellipsoidal bases in higher dimensions. Indeed, the two regimes in \eqref{e27}, \eqref{e28} (for $d=2$)  provide bounds for the value of a critical eccentricity, where the changeover from a maximum at the base to a maximum at the lateral side of the cylinder would take place. Several quantities may be needed to give the complete changeover picture for $d>2$.

Conditions \eqref{e27} and \eqref{e28} and the conclusions in Theorem \ref{the1} are invariant under homotheties: that is replacing $a_i,\,i=1,...,d$ and $L$ by $\lambda a_i,\,i=1,...,d$ and $\lambda L$ respectively, where $\lambda>0,$ and where the homothety is with respect to the centre of the base.

In Proposition \ref{prop1} below we give bounds for $I_d$ which guarantee that the collection of ellipsoids in Theorem \ref{the1}(i) is non-empty for $d\ge 2$. We adopt the standard notation $J_{\nu},\,\nu\ge -1$ for the Bessel function of the first kind, and denote by $j_{\nu,1}<j_{\nu,2}<....$ its strictly positive zeros.
\begin{proposition}\label{prop1}
\begin{itemize}
\item[\textup{(i)}] If $d\ge 1$, then
\begin{equation}\label{e22}
I_d\le \Big(\frac{2}{\pi d}\Big)^{1/2},
\end{equation}
\item[\textup{(ii)}]
\begin{equation}\label{e22a}
I_1=\frac{8G}{\pi^2}<1.
\end{equation}
\item[\textup{(iii)}]
\begin{equation}\label{e23}
\hspace{12mm}I_2=\sum_{k=1}^{\infty}\frac{2}{j^2_{0,k}J_1(j_{0,k})}>0.50976,
\end{equation}
\begin{equation}\label{e24}
I_3=\frac{2\log 2}{\pi},
\end{equation}
and
\begin{equation} \label{e25}
I_d>\Big(\frac{d}{2}\Big)^{-1/2}\Big(\Big(\frac{d}{2}+1\Big)^{1/2}+1\Big)^{-1},\, d\ge 4.
\end{equation}
Moreover,
\begin{equation}\label{e26}
I_d> \frac1d,\,d\ge 2.
\end{equation}
\end{itemize}
\end{proposition}

The torsion function is closely connected to Brownian motion. Tools of Brownian motion have been used in e.g. \cite{Bu} to prove the existence of maximisers of shape optimisation problems involving the maximum of the length of the gradient of the torsion function. This connection is as follows.
Let $(B(s),\,s\ge 0; \Pa_x,\,x\in \R^d)$ be Brownian motion with generator $\Delta$, and let
\begin{equation}\label{e1}
T(\Om)=\inf\{s\ge 0: B(s)\not\in \Om\}
\end{equation}
be the lifetime of Brownian motion in $\Om$. It is well known \cite{PS} that
\begin{equation}\label{e2}
\Pa_x[T(\Om)>t]=v_{\Om}(x;t),
\end{equation}
where $v_{\Om}$ is the solution of \eqref{e3}.
The expected value of the random variable $T(\Om)$ is given by
\begin{equation}\label{e4}
\mathbb{E}_x[T(\Om)]=\int_{(0,\infty)}\dd t\,\Pa_x[T(\Om)>t]=\int_{(0,\infty)}\dd t\,v_{\Om}(x;t)=w_{\Om}(x).
\end{equation}

\smallskip

Let $p_{\Om}(x,y;t),\,x\in\Om,\,y\in\Om,\,t>0$ denote the Dirichlet heat kernel for $\Om$. This heat kernel is non-negative, symmetric in $x$ and $y$, and satisfies the heat semigroup property.  If the spectrum of the Dirichlet Laplacian for $\Om$ is discrete, then the Dirichlet heat kernel has an $L^2(\Om)$ eigenfunction expansion. Moreover,
\begin{equation}\label{e8}
v_{\Om}(x;t)=\int_{\Om} \dd y\, p_{\Om}(x,y;t).
\end{equation}
By \eqref{e3}, \eqref{e4} and \eqref{e8} we have that
\begin{equation*}
w_{\Om}(x)=\int_{(0,\infty)}\dd t\int_{\Om} \dd y\, p_{\Om}(x,y;t).
\end{equation*}

\smallskip

It is of great interest to generalise the results for ellipsoids in Theorem \ref{the1} to general convex sets.  Let $C$ be convex in $\R^2$ with inradius $r(C)$, and circumscribed radius $R(C)$ respectively. One would expect that if $$ \frac{R(C)}{r(C)}<1+\varepsilon,$$
with $\varepsilon$ sufficiently small, then this would correspond to ellipsoids with small eccentricity, and the maximum length of the gradient is at the base. Since $\varepsilon$ is small the centres of the inball and circumscribed ball respectively are close. This suggests that the maximum of the length of the gradient is attained near the centre of the inball of $C$. Replacing $E(a)$ by $C$ and $a_{-}B_d$ by $r(C)B_d$ in \eqref{e78} below provides a lower bound for that maximum. For an upper bound one might use that $(0,L)\times C\subset \R\times C$. Lemma \ref{lem0} below then yields that for any point $x_0$ on the lateral side of the cylinder $|\nabla w_{(0,L)\times C}|(x_0)\le \|\nabla w_C\|_{\infty}$. In the case where $C$ is an ellipse we use the explicit expression for that $L^{\infty}$ norm.
The proof in Section \ref{sec3} then makes use of the fact that $2I_2=1.195..>1$, which indicates a fine margin for estimating $\|\nabla w_C\|_{\infty}$ for more general $C$.

We have the following monotonicity result.
\begin{lemma}\label{lem0}
Let $\Omega_1$ and $\Omega_2$ be convex sets such that $\Omega_1\subset \Omega_2\subset \R^d$. If there exists $x_0\in \partial \Omega_1 \cap \partial \Omega_2$ such that
$\partial\Om_1$ is smooth in a neighbourhood of $x_0$,
then
\begin{equation}\label{ex1}
|\nabla w_{\Om_1}|(x_0)\le |\nabla w_{\Om_2}|(x_0).
\end{equation}
\end{lemma}

\begin{proof}
Since $\Om_1\subset \Om_2$ and since  $\partial \Om_1$ is smooth in a neighbourhood of $x_0$,  $\partial \Om_2$ is smooth in a neighbourhood of $x_0$.
Let $\nu_0$ be the unit vector normal to $\partial \Omega_1$ at $x_0$, pointing inwards. Then
\begin{equation}\label{ex2}
|\nabla w_{\Omega_1}| (x_0) = \lim_{t\to 0} \frac{w_{\Omega_1}(x_0+t\nu) - w_{\Omega_1}(x_0)}{t} = \lim_{t\to 0} \frac{w_{\Omega_1}(x_0+t\nu)}{t}.
\end{equation}
Here we have used the fact that the torsion function vanishes on the boundary.

Note that the unit vector to $\partial \Omega_2$ at $x_0$ is again $\nu_0$. Hence
\begin{equation}\label{ex3}
|\nabla w_{\Omega_2}| (x_0)= \lim_{t\to 0} \frac{w_{\Omega_2}(x_0+t\nu) }{t}\,.
\end{equation}
By the maximum principle, we infer that $w_{\Omega_1}\leq w_{\Omega_2}$ in the common domain $\Omega_1$. In particular, for all $t$ small,
\begin{equation}\label{ex4}
w_{\Omega_1}(x_0+t\nu) \leq w_{\Omega_2}(x_0+t\nu).
\end{equation}
Inequality \eqref{ex1} follows by \eqref{ex2},\eqref{ex3} and \eqref{ex4}.
\end{proof}

\medskip

The proofs are organised as follows. In Section \ref{sec2} we obtain bounds (Lemma \ref{lem1}) for the length of the gradient of the torsion function of $(0,L)\times C$ at $C$, where $C$ is an open set with $\lambda(C)>0$. In Sections \ref{sec3} and \ref{sec4} we prove Theorem \ref{the1}(i) and (ii) respectively. In Section  \ref{sec5} we prove Proposition \ref{prop1}.

\section{Gradient of the torsion function at the base of a cylinder \label{sec2}}

In Lemma \ref{lem1} below we obtain bounds for the gradient of the torsion function of a cylinder $(0,L)\times C$ at an arbitrary point in its base. We see that, for $L$ large, the leading term is determined by a function which depends on $C$ only, and which is independent of $L$. The lower bound in \eqref{e29} will be used in the proof of Theorem \ref{the1}(i) in Section \ref{sec3}. There it will be shown that, for $C=E_a$ and $x'$ in the centre of $E_a$ this lower bound is strictly larger than the length of the gradient at any point at the lateral side of that cylinder. The upper bound in \eqref{e29} will be used in the proof of Theorem \ref{the1}(ii) in Section \ref{sec4}. There it will be shown that there are two points on the lateral side of the cylinder which have a gradient with length larger than the maximum length of the gradient at the base $E_a$.

We denote, with slight abuse of notation, the torsion function for open sets $C\subset \R^d$ by $w_C$, the one-dimensional Dirichlet heat kernel for the interval $(0,L)$ by $p_{(0,L)}(x_1,y_1;t),\,x_1\in (0,L),\,y_1\in (0,L),\,t>0$, and the $d$-dimensional Dirichlet heat kernel for $C$ by $p_C(x',y';t),\,x'\in C,\,y'\in C,\,t>0$.

\begin{lemma}\label{lem1} Let $d\ge 2$, and let $C$ be an open subset of $\R^d$ with $\lambda(C)>0$. If $L>0$, then
\begin{align}\label{e29}
\int_{(0,\infty)}\frac{\dd t}{(\pi t)^{1/2}}v_{C}(x';t)-\frac{4w_C(x')}{L}&\le |\nabla w_{(0,L)\times C}|(0,x')\nonumber\\&\le \int_{(0,\infty)}\frac{\dd t}{(\pi t)^{1/2}}v_C(x';t)\nonumber\\&
\le \frac{2}{\pi^{1/2}}\big(w_C(x')\big)^{1/2},\,x'\in C.
\end{align}
\end{lemma}

\begin{proof}
The one-dimensional heat kernel is given by its $L^2$- eigenfunction expansion
\begin{equation*}
p_{(0,L)}(x_1,y_1;t)=\frac{2}{L}\sum_{k=1}^{\infty}e^{-t\pi^2k^2/L^2}\sin(\pi kx_1/L)\sin(\pi ky_1/L).
\end{equation*}
Hence
\begin{align}\label{e31}
v_{(0,L)}(x_1;t)&=\frac{2}{L}\sum_{k=1}^{\infty}e^{-t\pi^2k^2/L^2}\sin(\pi kx_1/L)\int_{(0,L)}\dd y_1\,\sin(\pi ky_1/L)\nonumber\\
&=\frac{4}{\pi}\sum_{k=1,3,...}^{\infty}k^{-1}e^{-t\pi^2k^2/L^2}\sin(\pi kx_1/L).
\end{align}
Since
\begin{equation*}
v_C(x';t)=\int_C \dd y'\,p_C(x',y';t),
\end{equation*}
we have by \eqref{e31}
\begin{align*}
w_{(0,L)\times C}(x_1,x')&=\int_{(0,\infty)}\dd t\,v_{(0,L)}(x_1;t)v_C(x';t)\nonumber\\&
=\int_{(0,\infty)}\dd t\frac{4}{\pi}\sum_{k=1,3,...}^{\infty}k^{-1}e^{-t\pi^2k^2/L^2}\sin(\pi kx_1/L)v_C(x';t).
\end{align*}
Using Lebesgue's dominated convergence theorem we differentiate with respect to $x_1$ under the integral, and find that
\begin{equation}\label{e34}
\Big|\frac{\partial w_{(0,L)\times C}}{\partial x_1}\Big|(0,x')=\int_{(0,\infty)}\dd t\,\frac{4}{L}\sum_{k=1,3,...}^{\infty}e^{-t\pi^2k^2/L^2}v_C(x';t),\,\,x'\in C.
\end{equation}
We have that
\begin{align}\label{e35}
\frac{4}{L}\sum_{k=1,3,...}^{\infty}e^{-t\pi^2k^2/L^2}&\ge \frac{4}{L}\sum_{k=2,4,...}^{\infty}e^{-t\pi^2k^2/L^2}\nonumber\\&
=\frac{4}{L}\sum_{k=1,2,...}^{\infty}e^{-4t\pi^2k^2/L^2}\nonumber\\&
\ge \frac{4}{L}\Big(\int_{(0,\infty)}\dd k\,e^{-4t\pi^2k^2/L^2}-1\Big)\nonumber\\&
=\frac{1}{(\pi t)^{1/2}}-\frac{4}{L}.
\end{align}
By \eqref{e34} and \eqref{e35}
\begin{align*}
\Big|\frac{\partial w_{(0,L)\times C}}{\partial x_1}\Big|(0,x')&\ge \int_{(0,\infty)}\dd t\,\Big(\frac{1}{(\pi t)^{1/2}}-\frac{4}{L}\Big)v_C(x';t)\nonumber\\&
=\int_{(0,\infty)}\frac{\dd t}{(\pi t)^{1/2}}v_C(x';t)-\frac{4}{L}w_C(x'),\,x'\in C.
\end{align*}
This proves the lower bound in \eqref{e29}.

To prove the upper bounds in \eqref{e29} we have that
\begin{align}\label{e37}
\frac{4}{L}\sum_{k=1,3,...}^{\infty}e^{-t\pi^2k^2/L^2}&=\frac4L\sum_{k\in \N}\Big(e^{-t\pi^2k^2/L^2}-e^{-4t\pi^2k^2/L^2}\Big)\nonumber\\&
=\frac2L\sum_{k\in \Z}\Big(e^{-t\pi^2k^2/L^2}-e^{-4t\pi^2k^2/L^2}\Big)\nonumber\\&
=\frac{2}{(\pi t)^{1/2}}\sum_{k\in \Z}e^{-k^2L^2/t}-\frac{1}{(\pi t)^{1/2}}\sum_{k\in \Z}e^{-k^2L^2/(4t)}\nonumber\\&
=\frac{1}{(\pi t)^{1/2}}+\frac{2}{(\pi t)^{1/2}}\Big(\sum_{k\in \N}e^{-k^2L^2/t}-\sum_{k=1,3,...}e^{-k^2L^2/(4t)}\Big)\nonumber\\&
\le\frac{1}{(\pi t)^{1/2}}+\frac{2}{(\pi t)^{1/2}}\Big(\sum_{k\in \N}e^{-k^2L^2/t}-\sum_{k=2,4,...}e^{-k^2L^2/(4t)}\Big)\nonumber\\&
=\frac{1}{(\pi t)^{1/2}},
\end{align}
where we have used \cite[p.444]{ECT}
\begin{equation*}
\sum_{k\in \Z}e^{-k^2x^2}=\frac{\pi^{1/2}}{x}\sum_{k\in\Z}e^{-k^2\pi^2/x^2},\,x>0.
\end{equation*}
By Lemma \ref{lem0} with $\Om_1=(0,L)\times C,\,\Om_2=(0,\infty)\times C,\,x_0=(0,x')$, \eqref{e34} and \eqref{e37} we have that
\begin{align}\label{e39}
\Big|\frac{\partial w_{(0,L)\times C}}{\partial x_1}\Big|(0,x')&\le \Big|\frac{\partial w_{(0,\infty)\times C}}{\partial x_1}\Big|(0,x')\nonumber\\&
\le\int _{(0,\infty)}\frac{\dd t}{(\pi t)^{1/2}}v_C(x';t)\nonumber\\&
=\int _{(0,\infty)}\frac{\dd t}{(\pi t)^{1/2}}\Pa_{x'}(T_C>t)\nonumber\\&
=\int _{(0,\infty)}\frac{\dd t}{(\pi t)^{1/2}}\mathbb{E}_{x'}\Big[1_{\{T_C>t\}}\Big]\nonumber\\&
=\mathbb{E}_{x'}\Big[\int _{(0,\infty)}\frac{\dd t}{(\pi t)^{1/2}}1_{\{T_C>t\}}\Big]\nonumber\\&
=\frac{2}{\pi^{1/2}}\mathbb{E}_{x'}[T^{1/2}_C]\nonumber\\&
\le \frac{2}{\pi^{1/2}}\Big(\mathbb{E}_{x'}[T_C]\Big)^{1/2}\nonumber\\&=\frac{2}{\pi^{1/2}}\big(w_C(x')\big)^{1/2},\,x'\in C,
\end{align}
where we have used \eqref{e2} in the second line and where we have written $\Pa_{x'}$ as an expectation $\mathbb{E}_{x'}$ of an indicator function in the third line.
In the fourth line we have used Tonelli's theorem, and in the sixth line we have used Cauchy Schwarz's inequality. The right-hand side of \eqref{e39} is finite by the hypothesis on $C$. This completes the proof of \eqref{e29}.
\end{proof}

\section{Proof of Theorem \ref{the1}(i) \label{sec3}}

In this section we will prove Theorem \ref{the1}(i).
By monotonicity  of $\Om\mapsto w_{\Om}$, and the fact that $E_a\supset a_{-} B_d$ we have by the lower bound in \eqref{e29} that
\begin{equation}\label{e78}
|\nabla w_{(0,L)\times E_a}|(0,0)\ge \int_{(0,\infty)}\frac{\dd t}{(\pi t)^{1/2}}v_{a_{-} B_d}(0;t)-\frac{4w_{E_a}(0)}{L}.
\end{equation}
By scaling of the heat equation
\begin{equation}\label{e79}
v_{a_{-} B_d}(0;t)=v_{ B_d}(0;t/a^2_{-}).
\end{equation}
Moreover,
\begin{equation}\label{e80}
w_{E_a}(0)=\frac12\Big(\sum_{i=1}^{d} \Big(\frac{1}{a_i}\Big)^2\Big)^{-1}.
\end{equation}
Combining \eqref{e78}-\eqref{e80} gives that
\begin{equation*}
|\nabla w_{(0,L)\times E_a}|(0,0)\ge a_{-} I_d-\frac{2}{L}\Big(\sum_{i=1}^{d} \Big(\frac{1}{a_i}\Big)^2\Big)^{-1}.
\end{equation*}
By Lemma \ref{lem0} with $\Om_1=(0,L)\times E_a,\,\Om_2=\R \times E_a,\,x_0=(x_1,x')$ and $0<x_1<L,$ we have that
\begin{equation*}
|\nabla w_{(0,L)\times E_a}|(x_1,x')\le \frac{1}{a_{-}}\Big(\sum_{i=1}^{d} \Big(\frac{1}{a_i}\Big)^2\Big)^{-1}.
\end{equation*}
We conclude that if
\begin{equation*}
\frac{1}{a_{-}}\Big(\sum_{i=1}^{d} \Big(\frac{1}{a_i}\Big)^2\Big)^{-1}<a_{-} I_d-\frac{2}{L}\Big(\sum_{i=1}^{d} \Big(\frac{1}{a_i}\Big)^2\Big)^{-1},
\end{equation*}
then $|\nabla w_{(0,L)\times E_a}|$ has maxima  at $x_1=0$ and $x_1=L$. This is equivalent to \eqref{e27}.

\smallskip

It remains to show that the maxima are located at $(0,0)$ and $(L,0)$ respectively. For this we use \eqref{e34}.
It was shown in  \cite[Proposition 4.1(i)]{BMS} that if $C$ is convex, for fixed $t$, $v_C(x;t)$ is maximal for some $x_t$ in the heart of $C$.  If $C$ has $d$ planes of symmetry then the heart coincides with the centre of $C$ \cite[Theorem 2.4(i)]{BM}. In that case the location of the maximum is independent of $t$.
In case $C=E_a$ we conclude $v_{E_a}(x';t)\le v_{E_a}(0;t),\,t>0$. Hence the left-hand side of \eqref{e34} has a maximum at $(0,0)$. By symmetry $(L,0)$ is the second maximum. This proves \eqref{e27a}.
\hfill$\square$

\medskip

\section{Proof of Theorem \ref{the1}(ii) \label{sec4}}

To prove Theorem \ref{the1}(ii) we use the upper bound in \eqref{e29} for the length of the gradient of the torsion function for $(0,L)\times E_a$ at its base.
Since $E_a$ is contained in the region bounded by two parallel planes at distance $2a_{-}$ we use monotonicity to conclude that
\begin{align}\label{e85}
|\nabla w_{(0,L)\times E_a}|(0,x')&\le \int_{(0,\infty)}\frac{\dd t}{(\pi t)^{1/2}}v_{(-a_{-},a_{-})}(x';t)\nonumber\\&
\le  \int_{(0,\infty)}\frac{\dd t}{(\pi t)^{1/2}}v_{(-a_{-},a_{-})}(0;t)\nonumber\\&
=a_{-}\int_{(0,\infty)}\frac{\dd t}{(\pi t)^{1/2}}v_{(-1,1)}(0;t)\nonumber\\&
=a_{-}I_1\nonumber\\&
=\frac{8Ga_{-}}{\pi^2},
\end{align}
by Proposition \ref{prop1}(ii).

\smallskip

Next we obtain a lower bound for gradient of the torsion function for $x'\in \partial E_a$ with $x_1=\frac{L}{2}$.
Since the Dirichlet heat kernel for $(0,L)\times E_a$ factorises we have with $x=(x_1,x'),\,y=(y_1,y')$
\begin{equation*}
p_{(0,L)\times E_a}(x,y;t)=p_{E_a}(x',y';t)p_{(0,L)}(x_1,y_1;t).
\end{equation*}
Integrating with respect to $y$ over $(0,L)\times E_a$ gives that
\begin{align*}
v_{(0,L)\times E_a}(x;t)&=v_{E_a}(x';t)v_{(0,L)}(x_1;t)\nonumber\\&
=v_{E_a}(x';t)-\Big(1-v_{(0,L)}(x_1;t)\Big)v_{E_a}(x';t).
\end{align*}
A further integration with respect to $t$ over $[0,\infty)$ yields
\begin{equation}\label{e89}
w_{(0,L)\times E_a}(x)=w_{E_a}(x')-\int_0^\infty \dd t\,\Big(1-v_{(0,L)}(x_1;t)\Big)v_{E_a}(x';t).
\end{equation}
Since the left-hand side of \eqref{e89} is symmetric in $x$ with respect to $x_1=\frac{L}{2}$, and the first term in the right-hand side of \eqref{e89} is independent of $x_1$, the second term in the right-hand side of
 \eqref{e89} is symmetric in $x$ with respect to $x_1=\frac{L}{2}$. Taking gradients of \eqref{e89} gives that
\begin{equation*}
|\nabla w_{(0,L)\times E_a}|(L/2,x')\ge |\nabla w_{E_a}(x')|-\Big|\nabla \int_0^\infty \dd t\,\Big(1-v_{(0,L)}(x_1;t)\Big)v_{E_a}(x';t)\Big|
\end{equation*}
We infer that for any point $x$ with $x_1=\frac{L}{2}$ on the lateral boundary of the cylinder
\begin{align}\label{e91}
|\nabla w_{(0,L)\times E_a}|(L/2,x')&\ge |\nabla w_{E_a}|(x')-\Big|\nabla \int_0^\infty \dd t\,\Big(1-v_{(0,L)}(L/2;t)\Big)v_{E_a}(x';t)\Big|\nonumber\\&
=|\nabla w_{E_a}|(x')- \Big|\int_0^\infty \dd t\,\Big(1-v_{(0,L)}(L/2;t)\Big)\nabla v_{E_a}(x';t)\Big|\nonumber\\&
\ge |\nabla w_{E_a}|(x')- \int_0^\infty \dd t\,\Big(1-v_{(0,L)}(L/2;t)\Big)|\nabla v_{E_a}|(x';t).
\end{align}

We choose $x'=\tilde{x}'$ such that $|\nabla w_{E_a}|(x')$ is maximal. So
by \eqref{e18} and \eqref{e91}
\begin{align}\label{e93}
|\nabla w_{(0,L)\times E_a}|&(L/2,\tilde{x}')\nonumber\\&\ge\frac{1}{a_{-}}\Big(\sum_{i=1}^d\frac{1}{a_i^2}\Big)^{-1}-\int_0^\infty \dd t\,\Big(1-v_{(0,L)}(L/2;t)\Big)|\nabla v_{E_a}|(\tilde{x}';t).
\end{align}

Since $T(\Om)$ in \eqref{e1} is monotone under set inclusions  $v_{\Om}(\cdot,t)$ is monotone under set inclusions by \eqref{e2}. Hence we have monotonicity of $\nabla v_{\Om}(\cdot,t)$ under set inclusion as in Lemma \ref{lem0}. Since  $E_a$ is contained in the region bounded by two parallel hyper planes at distance $2a_{-}$ tangent at $\tilde{x}'$ and $-\tilde{x}'$ respectively,
\begin{align}\label{e94}
|\nabla v_{E_a}|(\tilde{x}';t)&\le\Big|\frac{\partial }{\partial z}
\frac{4}{\pi}\sum_{n=0}^\infty \frac{(-1)^n}{2n+1}e^{-(2n+1)^2\pi^2t/(4a_{-}^2)}\cos\frac{(2n+1)\pi z_1}{2a_{-}}\Big|_{z_1=-a_{-}}\nonumber\\&
=\frac{2}{a_{-}}\sum_{n=0}^\infty e^{-(2n+1)^2 \pi^2t/(4a_{-}^2)}\nonumber\\&
\le \frac{2}{a_{-}}e^{-\pi^2t/(8a_{-}^2}\sum_{n=0}^\infty e^{-(2n+1)^2\pi^2t/(8a_{-}^2)}\nonumber\\&
\le \Big(\frac{2}{\pi t}\Big)^{1/2}e^{-\pi^2t/(8a_{-}^2)},
\end{align}
where we have used \eqref{e37} with $L=2^{3/2}a_{-}$ in the last inequality.
By \cite[Lemma 4]{MvdBSri} we have that
\begin{equation}\label{e95}
1-v_{(0,L)}(L/2;t)\le 2^{3/2}e^{-L^2/(32t)}.
\end{equation}
By \eqref{e93}--\eqref{e95}
\begin{align}\label{e96}
|\nabla w_{(0,L)\times E_a}|(L/2,\tilde{x}')&\ge\frac{1}{a_{-}}\Big(\sum_{i=1}^d\frac{1}{a_i^2}\Big)^{-1}-\int_0^{\infty}\frac{4\dd t}{(\pi t)^{1/2}}e^{-L^2/(32t)-\pi^2t/(8a_{-}^2)}\nonumber\\&
=\frac{1}{a_{-}}\Big(\sum_{i=1}^d\frac{1}{a_i^2}\Big)^{-1}-\frac{2^{7/2}a_{-}}{\pi}e^{-\pi L/(8a_{-})},
\end{align}
where we have used the change of variable $t=\theta^2$, and the integral identity \cite[3.325]{GR} in the last equality. If the right-hand side of \eqref{e96} is strictly larger than the right-hand side of \eqref{e85}, that is if \eqref{e28} holds, then the maxima of the length of the gradient are located on the lateral of the cylinder.
Requirement \eqref{e28} follows by \eqref{e85} and \eqref{e96}. This implies \eqref{e28a}.
\hfill$\square$

It is an open problem to show that if \eqref{e28} is satisfied, then the only maxima of $|\nabla w_{(0,L)\times E_a}|$ are the points $(\frac{L}{2},\pm\tilde{x}').$

\section{Proof of Proposition \ref{prop1} \label{sec5}}
{\it Proof of Proposition \ref{prop1}\textup{(i)}.}
First note that the integral converges by \eqref{e39}.
Furthermore since $w_{B_d}(x')\le w_{B_d}(0)=\frac{1}{2d}$, \eqref{e39} gives \eqref{e22}.

\medskip

\noindent{\it Proof of Proposition \ref{prop1}\textup{(ii)},\,$d=1$}.
By \eqref{e31} we obtain that
$$ v_{B_1}(0;t)=v_{(0,2)}(1;t)=\frac{4}{\pi}\sum_{k=1,3,...}^{\infty}k^{-1}(-1)^{(k-1)/2}e^{-t\pi^2k^2/4}.$$
Multiplying this  equality with $(\pi t)^{-1/2}$, and integrating with respect to $t$ over $(0,\infty)$ gives \eqref{e22a} by definition \eqref{e28b}.

\medskip

\noindent{\it Proof of Proposition \ref{prop1}\textup{(iii)},\,$d\ge 4$}.
Let $\Om$ have strictly positive and finite measure.  Denote the spectrum of the Dirichlet Laplacian acting in $L^2(\Om)$ by $\lambda_{1,\Om}\le \lambda_{2,\Om}\le ...$
accumulating at infinity only. Let $\{\varphi_{1,\Om},\varphi_{2,\Om},...\}$ be a corresponding orthonormal basis of eigenfunctions. We choose $\varphi_{1,\Om}\ge 0$, and recall that
$\varphi_{1,\Om}\in L^{\infty}(\Om)$. See for example \cite[Theorem 1.4]{MvdBEB}. Furthermore the Dirichlet heat kernel
has an $L^2(\Om)$ eigenfunction expansion given by
\begin{equation*}
p_{\Om}(x,y;t)=\sum_{j\in \N}e^{-t\lambda_{j,\Om}}\varphi_{j,\Om}(x)\varphi_{j,\Om}(y).
\end{equation*}
By the orthonormality of the eigenfunctions and Fubini's theorem
\begin{align}\label{e41}
v_{\Om}(x;t)=\int_{\Om}\dd y\,p_{\Om}(x,y;t)&=\int_{\Om}\dd y\,\sum_{j\in \N}e^{-t\lambda_{j,\Om}}\varphi_{j,\Om}(x)\varphi_{j,\Om}(y)\nonumber\\&\ge
\int_{\Om}\dd y\,\sum_{j\in \N}e^{-t\lambda_{j,\Om}}\varphi_{j,\Om}(x)\varphi_{j,\Om}(y)\frac{\varphi_{1,\Om}(y)}{\|\varphi_{1,\Om}\|_{\infty}}\nonumber\\&
=e^{-t\lambda_{1,\Om}}\frac{\varphi_{1,\Om}(x)}{\|\varphi_{1,\Om}\|_{\infty}}.
\end{align}
Multiplying both sides of \eqref{e41} with $(\pi t)^{-1/2}$ and integration with respect to $t$ over $(0,\infty)$ gives that
\begin{align}\label{e42}
\int_{(0,\infty)}\frac{\dd t}{(\pi t)^{1/2}}v_{\Om}(x;t)&\ge\int_{(0,\infty)}\frac{\dd t}{(\pi t)^{1/2}}e^{-t\lambda_{1,\Om}}\frac{\varphi_{1,\Om}(x)}{\|\varphi_{1,\Om}\|_{\infty}}\nonumber\\&
=\frac{1}{(\lambda_{1,\Om})^{1/2}}\frac{\varphi_{1,\Om}(x)}{\|\varphi_{1,\Om}\|_{\infty}}.
\end{align}
Substitution of $\Om=B_d$ and $x=0$ in \eqref{e42} yields
\begin{equation}\label{e43}
I_d\ge \frac{1}{(\lambda_{1,B_d})^{1/2}}=\frac{1}{j_{\frac{d}{2}-1,1}}.
\end{equation}
It was shown  in \cite[(2)]{Ch} that
\begin{equation*}
j_{\frac{d}{2}-1,1}<\Big(\frac{d}{2}\Big)^{1/2}\Big(\Big(\frac{d}{2}+1\Big)^{1/2}+1\Big).
\end{equation*}
This gives by \eqref{e43} that
\begin{equation*}
I_d>\Big(\frac{d}{2}\Big)^{-1/2}\Big(\Big(\frac{d}{2}+1\Big)^{1/2}+1\Big)^{-1}.
\end{equation*}
This proves \eqref{e25}.

\medskip

To prove the remaining assertions of Proposition \ref{prop1} we need the following.
\begin{lemma}\label{lem2}
Let $\Om=B_d$, and denote the solution of \eqref{e3} by $v_{B_d}$. If $d=1,2,3,...$, then
\begin{equation}\label{e46}
v_{B_d}(x;t)=\sum_{k=1}^{\infty}\frac{2r^{-\nu}}{j_{\nu,k}J_{\nu+1}(j_{\nu,k})}J_{\nu}(j_{\nu,k}r)e^{-j_{\nu,k}^2t},
\end{equation}
where $r=|x|$, and
\begin{equation*}
\nu=\frac{d}{2}-1.
\end{equation*}
\end{lemma}
\begin{proof}
Since both initial data and boundary conditions have radial symmetry we only consider the radial solutions and rewrite \eqref{e3} for $\Om=B_d$ with $v=v(r;t)$ as
\begin{equation}\label{e48}
\begin{cases}
\frac{\partial v}{\partial t}=\Delta_r v\qquad &r\in [0,1), t>0,\\
v(1;t)=0, v(0;t)<\infty,\qquad &t>0,\\
v(r;0)=1,\qquad  &r\in [0,1),
\end{cases}
\end{equation}
where
\begin{equation*}
\Delta_r=\frac{1}{r^{d-1}}\frac{\partial}{\partial r} r^{d-1} \frac{\partial}{\partial r}
\end{equation*}
acts in $L^2\left((0,1); r^{d-1}\right)$.

We note that the eigenvalues and (not yet normalised but pairwise orthogonal) eigenfunctions of $-\Delta_r$ subject to the Dirichlet condition at $r=1$ and the regularity condition at $r=0$ are
\begin{equation*}
\lambda_k = j_{\nu, k}^2,\qquad U_k(r) = r^{-\nu}J_\nu\left( j_{\nu, k} r\right), \qquad k\in\mathbb{N}.
\end{equation*}

We seek the solution of \eqref{e48} by expanding in the eigenfunctions of $-\Delta_r$:
\begin{equation}\label{e51}
v(r;t)= \sum_{n=1}^\infty c_n U_n(r) \er^{-\lambda_n t}.
\end{equation}
Taking the initial condition $v(r;0)=1$, and multiplying it by $U_k$ in $L^2\left((0,1); r^{2\nu+1}\right)$, we get
\begin{equation}\label{e52}
c_k=\frac{\int_0^1 r^{2\nu+1} U_k(r)\,\dd r}{\int_0^1 r^{2\nu+1} (U_k(r))^2\,\dd r}=\frac{\int_0^1 r^{\nu+1}J_\nu\left( j_{\nu, k} r\right)\,\dd r}{\int_0^1 r J_\nu^2\left( j_{\nu, k} r\right)\,\dd r}\\
=\frac{\frac{1}{j_{\nu,k}}J_{\nu+1}(j_{\nu,k})}{\frac{1}{2}\left(J'_\nu(j_{\nu,k})\right)^2},
\end{equation}
where the integral in the numerator is evaluated by
\cite[11.3.20]{AS},
and the one in the denominator by \cite[11.4.5]{AS}. Using the recurrence relation \cite[9.1.27]{AS}
\begin{equation*}
J'_\nu(z) = -J_{\nu+1}(z)+\frac{\nu}{z}J_\nu(z)
\end{equation*}
evaluated at $z=j_{\nu, k}$, we deduce that
\begin{equation}\label{e54}
c_k=\frac{2}{j_{\nu,k} J_{\nu+1}(j_{\nu,k})}= -\frac{2}{j_{\nu,k} J'_{\nu}(j_{\nu,k})}.
\end{equation}
Formula \eqref{e46} follows from \eqref{e51}--\eqref{e54}.
\end{proof}

\smallskip

We continue with the proof of Proposition \ref{prop1}\textup{(iii)} for $d=3$. In that case $\nu=\frac12$, and by \cite[10.1.1,10.1.11]{AS} we find that
$j_{\frac12,k}=\pi k$. Evaluating the Bessel functions $J_{\frac12}$ and $J_{\frac32}$  gives
\begin{equation}\label{e55}
v_{B_3}(x;t)=2\sum_{k\in \N}(-1)^{k-1}\frac{\sin(k\pi r)}{{k\pi r}}e^{-k^2\pi^2 t},\, t>0.
\end{equation}
This jibes with \cite[p.233(4)]{CJ}, where the authors considered the problem with boundary condition $1$ on $\partial B_3$, and initial datum $0$ on $B_3$.
Replacing $v$ there by $1-v$ gives \eqref{e55}.

\smallskip

It is straightforward to verify that for $d=3$,
\begin{align}\label{e56}
\lim_{r\downarrow 0}v_{B_3}(x;t)&=2\sum_{k\in \N}(-1)^{k-1}e^{-k^2\pi^2 t}\nonumber\\&
=2\sum_{k\in \N}\Big(e^{-(2k-1)^2\pi^2 t}-e^{-(2k)^2\pi^2 t}\Big),\,t>0.
\end{align}
The summand of the right-hand side of \eqref{e56} is non-negative. By Tonelli's theorem and by \eqref{e21}, \eqref{e55}, \eqref{e56}
we have that
\begin{align*}
I_3&=\int_{(0,\infty)}\dd t\,\frac{1}{(\pi t)^{1/2}}\sum_{k\in \N}\Big(e^{-(2k-1)^2t}-e^{-(2k)^2t}\Big)\nonumber\\&
=2\sum_{k\in \N}\int_{(0,\infty)}\dd t\,\frac{1}{(\pi t)^{1/2}}\Big(e^{-(2k-1)^2\pi^2 t}-e^{-(2k)^2\pi^2 t}\Big)\nonumber\\&
=\frac{2}{\pi}\sum_{k\in \N}\Big((2k-1)^{-1}-(2k)^{-1}\Big)\nonumber\\&
=\frac{2\log 2}{\pi},
\end{align*}
where we  have used \cite[6.1.1]{AS} with $z=\frac12$, and
\begin{align*}
\sum_{k\in \N}\Big((2k-1)^{-1}-(2k)^{-1}\Big)&=\lim_{n\rightarrow \infty}\sum_{k=1}^n\Big((2k-1)^{-1}-(2k)^{-1}\Big)\nonumber\\&
=\lim_{n\rightarrow \infty}\sum_{k=1}^{n}\frac{(-1)^{k-1}}{k}=\log 2,
\end{align*}
in the last equality.
This proves \eqref{e24}.

\smallskip

We next consider the case $d=2$. Then $\nu=0$. This gives by \eqref{e46} that
\begin{equation*}
v_{B_2}(x;t)=2\sum_{k=1}^{\infty}\frac{J_0(rj_{0,k})}{j_{0,k}J_1(j_{0,k})}e^{-j_{0,k}^2t},\, t>0.
\end{equation*}
This jibes with \cite[p.199 (5)]{CJ}.
By \cite[9.1.7]{AS} with $\nu=0$ we find that
\begin{equation*}
\lim_{r\downarrow 0}v_{B_2}(x;t)=2\sum_{k=1}^{\infty}\frac{1}{j_{0,k}J_1(j_{0,k})}e^{-j_{0,k}^2t},\, t>0.
\end{equation*}
By \cite[9.1.27]{AS} $J_1(j_{0,k})=-J_0'(j_{0,k})$ so that
\begin{equation*}
v_{B_2}(0;t)=2\sum_{k=1}^{\infty}\frac{-1}{j_{0,k}J'_0(j_{0,k})}e^{-j_{0,k}^2t},\, t>0.
\end{equation*}
We first show that we may change the order of integration with respect to $t$ and summation over $k$. Define the partial sums
\begin{equation}\label{e62}
f_m(t)=2\sum_{k=1}^{m}\frac{-1}{j_{0,k}J'_0(j_{0,k})}\frac{e^{-j_{0,k}^2t}}{(\pi t)^{1/2}},\, t>0,\,m\in \N.
\end{equation}
In order to find an upper bound on the absolute value of the summand we require a lower bound for $|J'_0(j_{0,k})|$.
We write the Bessel function in terms of Bessel modulus and phase functions \cite[9.2.19]{AS},
\begin{equation*}
J_\nu(z)=M_0(z) \cos\theta_0(z).
\end{equation*}
Then
\begin{equation*}
J'_0(z)=M'_0(z) \cos\theta_0(z) - M_0(z) \theta'_0(z)\sin\theta_0(z),
\end{equation*}
and for $z\in Z_0:=\left\{j_{0,k}: k\in\mathbb{N}\right\}$ we have $\cos\theta_0(z) = 0$ and $|\sin\theta_0(z)|=1$.
Moreover, by \cite[9.2.21]{AS}
\begin{equation*}
M_0^2(x)\theta_0'(x)=\frac{2}{\pi x},\,x>0.
\end{equation*}
Hence $\theta_0'(j_{0,k})>0$, and
\begin{equation}\label{e66}
|J'_0(j_{0,k})|= M_0(j_{0,k}) \theta'_0(j_{0,k})=\frac{2}{\pi j_{0,k}M_0(j_{0,k})}.
\end{equation}
By \eqref{e62} and \eqref{e66},
\begin{align}\label{e67}
|f_m(t)|&\le \frac{4}{\pi}\sum_{k=1}^{m}M_0(j_{0,k})\frac{e^{-j_{0,k}^2t}}{(\pi t)^{1/2}}\nonumber\\&
\le \frac{4}{\pi}\sum_{k=1}^{\infty}M_0(j_{0,k})\frac{e^{-j_{0,k}^2t}}{(\pi t)^{1/2}}.
\end{align}
Since the right-hand side of \eqref{e67} is positive we have by Tonelli's theorem
\begin{align}\label{e68}
\int_{(0,\infty)}\dd t\,\sum_{k=1}^{\infty}M_0(j_{0,k})\frac{e^{-j_{0,k}^2t}}{(\pi t)^{1/2}}&\le \sum_{k=1}^{\infty}\int_{(0,\infty)}\dd t\,M_0(j_{0,k})\frac{e^{-j_{0,k}^2t}}{(\pi t)^{1/2}}\nonumber\\&
=\sum_{k=1}^{\infty}\frac{M_0(j_{0,k})}{j_{0,k}}.
\end{align}
We make the following observations: (i) $x\rightarrow M_0(x),\,x>0$ is decreasing by\cite[13.74]{Wat}, (ii)
\begin{equation*}
M_0(x)=\Big(\frac{2}{\pi x}\Big)^{1/2}(1+o(1)),\,x\rightarrow \infty,
\end{equation*}
by \cite[p.447(1)]{Wat},
and (iii)
\begin{equation*}
j_{0,k}=k\pi(1+o(1)),\, k\rightarrow\infty,
\end{equation*}
by \cite[p.504]{Wat}.  We conclude by (i)--(iii) that the series in the right-hand side of \eqref{e68} converges.
We conclude that the series $\sum_{k=1}^{\infty}\frac{-1}{j_{0,k}J'_0(j_{0,k})}e^{-j_{0,k}^2t}$ converges almost everywhere to a function $f$, and that,
by Lebesgue's dominated convergence theorem,
\begin{align*}
\int_{(0,\infty)}\dd t\, \sum_{k=1}^{\infty}\frac{-2}{j_{0,k}J'_0(j_{0,k})}\frac{e^{-j_{0,k}^2t}}{(\pi t)^{1/2}}=2\sum_{k=1}^{\infty}\frac{-1}{j^2_{0,k}J'_0(j_{0,k})}.
\end{align*}
This proves \eqref{e23} since $J_1(j_{0,k})=-J_0'(j_{0,k})$.

\smallskip

It remains to prove \eqref{e26} for $d=2$. Define the sequence
\begin{equation}\label{e72}
\left\{a_{k}\right\}_{k=1}^\infty:=\left\{-j_{0,k}^{2} J'_{0}\left(j_{0,k}\right)\right\},
\end{equation}
so that
\begin{equation*}
I_2=\sum_{k=1}^\infty\frac{2}{a_{k}}.
\end{equation*}
First observe that by the oscillatory nature of Bessel functions $\operatorname{sign} a_k = (-1)^{k+1}$. Next we show that
\begin{equation}\label{e74}
|a_k|<|a_{k+1}|,\, k\in \N,\,\, \lim_{k\rightarrow\infty}|a_{k}|= \infty.
\end{equation}
By \eqref{e66} and \eqref{e72}
\begin{equation*}
|a_k|=j_{0,k}^2|J'_0(j_{0,k})|=\frac{2j_{0,k}}{\pi M_0(j_{0,k})}.
\end{equation*}
Assertion \eqref{e74} follows since $k\rightarrow j_{0,k}$ is strictly increasing to $\infty$, and  $k\rightarrow M_0(j_{0,k})$ is strictly decreasing to $0$.

Since the sequence $\{a_k^{-1}\}_{k=1}^{\infty}$ has alternating signs, its absolute values are strictly decreasing to $0$. Moreover since $a_1>0$
we have that
\begin{equation*}
I_2>\sum_{k=1}^{2n} \frac{2}{a_{k}},\,n\in\N.
\end{equation*}
In particular
\begin{equation*}
I_2>\sum_{k=1}^4 \frac{2}{a_{k}}=2\sum_{k=1}^4 \frac{1}{j_{0,k}^2 J_1\left(j_{0,k}\right)}>0.50976,
\end{equation*}
which proves \eqref{e23}.
Finally, \eqref{e26} follows from \eqref{e23}--\eqref{e25}.

\hfill$\square$

\section*{Acknowledgements}
I am grateful to Michael Levitin for helpful comments, and for providing the non-trivial proof of the inequality in \eqref{e23} above.
I am grateful to Lorenzo Brasco, Rick Laugesen and the referees for their helpful comments and suggestions.

\section*{Statements and Declarations} The author declares that there is no conflict of interest. Data sharing is not applicable to this article as no datasets were generated or analyzed during the current study.

\section*{Funding Information}
The author would like to thank the Isaac Newton Institute for Mathematical Sciences, Cambridge, for support and hospitality during the programme Geometric Spectral Theory and applications, where work on this paper was undertaken. This work was supported by EPSRC grant EP/Z000580/1.



\begin{thebibliography} {99}

\bibitem{AS} Abramowitz, M., Stegun, I. A.: Pocketbook of Mathematical Functions. Material Selected by M. Danos, M., J. Rafelski. Verlag Harri Deutsch, Thun (1984)

\bibitem{B}  Beck, T., Brandolini, B., Burdzy, K., Henrot, A., Langford, J. L., Larson, S., Smits, R., Steinerberger, S.: Improved bounds for Hermite-Hadamard inequalities in higher dimensions.
J. Geometric Analysis \textbf{31} 801--816 (2021)

\bibitem{MvdBSri} van den Berg, M., Srisatkunarajah, S.: Heat flow and Brownian motion for a region in
${\mathbb R}^{2}$ with a polygonal boundary. Probability Theory and
Related Fields {\bf 86}, 41--52 (1990)

\bibitem{MvdBEB} van den Berg, M., Bolthausen, E.: Estimates for Dirichlet eigenfunctions. Journal of
the London Mathematical Society (2) {\bf 59}, 607--619 (1999)


\bibitem{vdBC} van den Berg, M., Carroll, T.: Hardy inequality and $L^p$
estimates for the torsion function. Bull. London Mathematical Society \textbf{41}, 980--986 (2009)
\bibitem{BR} Brasco, L., Ruffini, B.: Compact Sobolev embeddings and torsion functions. Ann. Institut H. Poincar\'e \textbf{34}, 817--843 (2017)

\bibitem{BMS} Brasco, L., Magnanini, R., Salani, P.: The location of the hot spot in a grounded  convex conductor. Indiana University Mathematics Journal \textbf{60}, 633--659 (2011)

\bibitem{BM} Brasco, L., Magnanini, R.: The heart of a convex body. In: Magnanini, R., Sakaguchi, S., Alvino, A.(eds) Geometric Propoerties for parabolic and elliptic PDE's. Springer INdAM Series, Springer, Milano \textbf{2} 49--66 (2013)

\bibitem{Bu} Burdzy, K., Ftouhi, I., Liu, X., Mariano, P.: Geometric properties of optimizers for the maximum gradient of the torsion function. arXiv:2512.09400 (2025)

\bibitem{CJ} Carslaw, H. S., Jaeger, J. C.: Conduction of heat in solids, Oxford Science Publications, Clarendon Press, Oxford (1992)

\bibitem{Ch} Chambers, Ll. G.: An upper bound for the first zero of Bessel functions.  Math. Comp. \textbf{38}, 589--591 (1982)

\bibitem{GR} Gradshteyn, I. S., Ryzhik, I. M.: Table of Integrals, Series and Products. Academic Press, San Diego (1994)

\bibitem{HS} Hoskins, J.G., Steinerberger, S.: Towards optimal gradient bounds for the torsion function in the plane. J. Geometric Analysis \textbf{31}, 7812--7841 (2021)

\bibitem{BK} Kawohl, B.: On the location of maxima of the gradient for solutions to quasi linear elliptic problems and a problem raised by Saint Venant. Journal of Elasticity \textbf{17}, 195--206 (1987)

\bibitem{SL} Larson, S.: A sharp multidimensional Hermite-Hadamard inequality.
Int. Math. Res. Not. IMRN, 1297--1312 (2022)

\bibitem{LY} Li, Q., Yao, R.: On location of maximum of gradient of torsion funcion. SIAM Journal on Mathematical Analysis \textbf{56}, 5372--5385 (2024)

\bibitem{LX} Li, Q., Xie, S., Yang, H., Yao, R.: On location of maxima gradient of torsion function over some nonsymmetric planar domains. arXiv:2406.04790v3 (2025)

\bibitem{P} Payne, L. E.: Bounds for solutions of a class of quasilinear elliptic boundary value problems in terms of the torsion function.
Proc. Royal Soc. Edinburgh \textbf{88A}, 251--265 (1981)

\bibitem{GP} P\'olya, G.: Liegt die Stelle der gr\"ossten Beanspruche an der Oberfl\"ache ? Z. Angew. Math. Mech. \textbf{10}, 353--360 (1930)

\bibitem{PS} Port, S. C., Stone, C. J.: Brownian motion and classical potential theory.
Academic Press, New York-London (1978)

\bibitem{Stb} Steinerberger, S.: The Hermite-Hadamard inequality in higher dimensions. J. Geometric Analysis \textbf{30}, 466--483 (2020)

\bibitem{HV} Vogt, H.: $L_{\infty}$-estimates for the torsion function and  $L_{\infty}$-growth of semigroups satisfying Gaussian bounds.
Potential Anal. \textbf{51}, 37--47 (2019)

\bibitem{RS} Sperb, R.: Maximum principles and their applications. Academic Press, New York (1981)

\bibitem{TG}Timoshenko, S.,  Goodier, J. N.: Theory of Elasticity. 2nd ed. McGraw-Hill,  New York (1951)

\bibitem{ECT} Titchmarsh, E. C.: The theory of functions. Oxford University Press, Oxford (1979)

\bibitem{Wat}
Watson, G. N.: A treatise on the theory of Bessel functions. 2nd ed.
Cambridge University Press, Cambridge (1995)


\bibitem{WH} Weinberger, H. F.: Remark on the preceding paper of Serrin. Arch. Rat. Mech. Anal. \textbf{43}, 319--320 (1971)

\end{thebibliography}
\end{document}